\documentclass[cropon,leqno,final,secnum]{rims-bessatsu}

\VolumeNo{x}           
\YearNo{201x}           
\PagesNo{000--000}      
\communication{Received April 20, 201x.
Revised September 11, 201x.}      

\title{Rigidity of pseudo-free group actions\\ on contractible manifolds}
\TitleHead{Rigidity of pseudo-free group actions}

\author{Qayum \textsc{Khan}\footnote{Department of Mathematics, University of Notre Dame du Lac, Notre Dame IN 46556 U.S.A.\newline e-mail: \texttt{qkhan@nd.edu}}}
\AuthorHead{Qayum Khan}

\support{Partially supported by the United States National Science Foundation (DMS-0904276)}

\usepackage{amsmath,amssymb,amscd}
\usepackage{eucal,mathrsfs,url}

\newtheorem{thm}{Theorem}[section]
\newtheorem{crl}[thm]{Corollary}

\newtheorem{lmm}[thm]{Lemma}

\theoremstyle{definition}

\newtheorem*{undefn}{Definition} 
\newtheorem{exa}[thm]{Example}
\theoremstyle{remark}

\DeclareMathAlphabet{\matheurm}{U}{eur}{m}{n}

\newcommand{\C}{\mathbb{C}}
\newcommand{\bbE}{\mathbb{E}}
\newcommand{\bbH}{\mathbb{H}}

\newcommand{\R}{\mathbb{R}}
\newcommand{\Z}{\mathbb{Z}}

\newcommand{\G}{\Gamma}

\newcommand{\cF}{\mathcal{F}}
\newcommand{\cG}{\mathcal{G}}

\newcommand{\fS}{\mathscr{S}}

\newcommand{\fbc}{\matheurm{fbc}}

\newcommand{\fin}{\matheurm{fin}}

\newcommand{\midd}{\matheurm{mid}}

\newcommand{\vc}{\matheurm{vc}}

\newcommand{\CAT}{\mathrm{CAT}}

\newcommand{\Hei}{\mathrm{Hei}}

\newcommand{\Stab}{\mathrm{Stab}}

\newcommand{\UNil}{\mathrm{UNil}}
\newcommand{\vcd}{\mathrm{vcd}}

\newcommand{\bdry}{\partial}

\newcommand{\vare}{\varepsilon}

\newcommand{\wt}[1]{\widetilde{#1}}

\begin{document}
\maketitle

In this article, we announce joint work with Frank Connolly and Jim Davis \cite{CDK2}.
This follows an earlier case study of pseudo-free involutions on the $n$-torus carried out in \cite{CDK}.  The author is grateful to the organizers of the RIMS conferences where these results were disseminated in Asia: \emph{Transformation Groups and Surgery Theory} (Masayuki Yamasaki, August 2010), \emph{Transformation Groups and Combinatorics} (Mikiya Masuda, June 2011).

\section{History and the Main Theorem}

\begin{undefn}
Let $\cF \subset \cG$ be families of subgroups of a group $\G$.
We say that $\G$ satisfies \emph{Property $C_{\cF \subset \cG}$} if every element $H \in \cG - \cF$ has its centralizer $C_\G(H)$ in $\cG$.
One says that $\G$ satisfies \emph{Property $M_{\cF \subset \cG}$} if every element $H \in \cG - \cF$ is contained in a unique maximal element $H_{max}$ of $\cG$.
Furthermore, one says that $\G$ satisfies \emph{Property $NM_{\cF \subset \cG}$} if $\G$ satisfies $M_{\cF \subset \cG}$ and each $H_{max}$ is self-normalizing in $\G$.
\end{undefn}

Below we consider the increasing chain $\{1\} \subset \fin \subset \fbc \subset \vc$ of families, where $\{1\}$ consists of the trivial subgroup, $\fin$ consists of the finite subgroups, $\fbc$ consists of the finite-by-cyclic subgroups, and $\vc$ consists of the virtually cyclic subgroups.

\begin{undefn}
Let $\G$ be a group. We define $\fS(\G)$ as the set of $\G$-homeomorphism classes of contractible manifolds equipped with an effective cocompact proper $\G$-action.
\end{undefn}

For any $\G$-space $X$, consider the free part of the action:
\[
X_{free} ~:=~ \{ x \in X \mid gx = x \text{ implies } g=1 \in \G \}.
\]

Our Main Theorem parameterizes $\fS(\G)$, and determines when it is one element.

\pagebreak
\begin{thm}[Main Theorem]\label{thm:mainQ}
Let $\G$ be a group. Assume:
\begin{enumerate}
\item $\G$ satisfies Property $C_{\{1\} \subset \fin}$,
\item $\G$ satisfies Property $M_{\fbc \subset \vc}$,
\item $\G$ is virtually torsion-free with $n := \vcd(\G) > 4$,
\item there exists $[X,\G] \in \fS(\G)$ where $X_{free}/\G$ has the homotopy type of a finite complex,
\item $\G$ satisfies the Farrell--Jones Conjecture in lower $K$-theory and in $L$-theory.
\end{enumerate}
Write $\vare := (-1)^n$.
There is a bijection of sets, with $0 \mapsto [X,\G]$, given by Wall realization:
\begin{equation}\label{eqn:main}
\bigoplus_{(\midd)(\G)} \UNil_{n+\vare}(\Z; \Z, \Z) ~\xrightarrow{~\approx~}~  \fS(\G).
\end{equation}
Here $(\midd)(\G)$ is the set of conjugacy classes of maximal infinite dihedral subgroups  of $\G$.\linebreak
Furthermore, each element of $\fS(\G)$ has a locally conelike representative with the same $\G$-homeomorphism type of links of singularities.

In particular, if $n\equiv 0,1 \pmod{4}$, or if $\G$ has no element of order two, then $\fS(\G)$ has only one element.
In this case, for any cocompact $\G$-manifold $M$, every $\G$-homotopy equivalence $f: M \to X$ is $\G$-homotopic to a $\G$-homeomorphism.
\end{thm}

For the proof, see the full article \cite{CDK2}. Notably, for the topological actions $\G \curvearrowright M$, Smith theory was used to get isolated fixed points from Hypothesis~(1), and Siebenmann theory was used to conclude the action must be locally conelike from Hypothesis~(4).

The vanishing result of the last paragraph of Theorem~\ref{thm:mainQ} is immediate from the following calculation \cite{CR, CD} of the Cappell groups that occur as the summands in \eqref{eqn:main}.

\begin{thm}[Connolly--Davis--Ranicki]
Let $n$ be an integer. Set $\vare := (-1)^n$. Then there is an isomorphism of abelian groups:
\[
\UNil_{n+\vare}(\Z; \Z, \Z) ~\cong~
\begin{cases}
0 & \text{if } n \equiv 0 \pmod{4}\\
0 & \text{if } n \equiv 1 \pmod{4}\\
(\Z/2)^\infty \oplus (\Z/4)^\infty & \text{if } n \equiv 2 \pmod{4}\\
(\Z/2)^\infty & \text{if } n \equiv 3 \pmod{4}.
\end{cases}
\]
\end{thm}

The parameterization of \eqref{eqn:main} is achieved away from the singularities by a smooth handle construction, where gluing instructions are given by generalized Arf invariants.

In Section~\ref{sec:cor}, we show that the above five properties are satisfied by certain actions on $\CAT(0)$ manifolds.
In Section~\ref{sec:exm}, we provide a family of exotic $\CAT(0)$ examples which cannot come from a Riemannian manifold of nonpositive sectional curvatures.

\section{Geometric consequences}\label{sec:cor}

\begin{undefn}
A proper action $\G \curvearrowright X$ is \emph{pseudo-free} if the singular set is discrete:
\[
X_{sing} ~:=~ \{ x \in X \mid gx=x \text{ for some } g \neq 1 \in \G \}.
\]
\end{undefn}

Theorem~\ref{thm:mainQ} was originally established in \cite{CDK} for the special case of the family of crystallographic groups $\G=\Z^n \rtimes_{-1} C_2$ for all $n > 3$. More generally, we conclude:

\begin{crl}\label{cor:EH}
Let $\G$ be a pseudo-free, cocompact, discrete group of isometries of Euclidean space $\bbE^n$ or hyperbolic space $\bbH^n$ with $n > 4$.
The bijection \eqref{eqn:main} holds.
\end{crl}

\begin{proof}
This will be immediate from Corollary~\ref{cor:CAT0} below and Selberg's lemma.
\end{proof}

Euclidean and hyperbolic spaces fit into a broader class, $\CAT(0)$ spaces (see \cite{BH_CAT0}):

\begin{crl}\label{cor:CAT0}
Let $X$ be a $\CAT(0)$ topological manifold of dimension $n > 4$.
Suppose $\G$ is a virtually torsion-free, locally conelike,
pseudo-free, cocompact discrete proper group of isometries of $X$.
Then the bijection \eqref{eqn:main} holds.
\end{crl}

This can be viewed as a generalization of \cite[Theorem~A]{BL_CAT0}.
Both rigidity results rely on the truth of the Farrell--Jones Conjecture for these groups, \cite[Theorem~B]{BL_CAT0}.

\begin{proof}
By assumption, Hypothesis~(3) holds.
Since any two points in $X$ are joined by a unique geodesic segment, $X$ is contractible.
Also, since $\G \curvearrowright X$ is locally conelike, the quotient $X_{free}/\G$ is the interior of a compact topological $\bdry$-manifold.
Hence Hypothesis~(4) holds.
By a recent theorem of Bartels--L\"uck \cite{BL_CAT0}, Hypothesis~(5) holds.

Let $H$ be a nontrivial finite subgroup of $\G$.
Since the action $\G \curvearrowright X$ is pseudo-free, the fixed set $X^H$ is a single point.
Note the proper action $\G \curvearrowright X$ restricts to a proper action $C_\G(H) \curvearrowright X^H$.
So $C_\G(H)$ is finite.
Therefore Hypothesis~(1) holds.

Let $D \in \vc(\G) - \fbc(\G)$.
There is a unique $D$-invariant geodesic line $\ell_D \subset X$, as follows.
It follows from Hypothesis~(1), see \cite{CDK2}, that $D$ is isomorphic to the infinite dihedral group, $D_\infty = C_2 * C_2$.
Write $D = \langle a, b \mid a^2 = b^2 = 1 \rangle$.
Let $x$ and $y$ be the unique fixed points in $X$ of $a$ and $b$.
Since $ab$ has infinite order, $x$ and $y$ are distinct, joined by a unique geodesic segment $\sigma \subset X$.
Note that $D\sigma$ is homeomorphic to $\R$ and is a closed subset of $X$.
Suppose $\ell \subset X$ is a $D$-invariant geodesic line.
Then $D\sigma \approx \R$ is a closed subset of $\ell \approx \R$.
Hence $D\sigma = \ell$.
Therefore, any such $\ell$ is unique.

It remains to prove such an $\ell$ exists, that is, the $D$-invariant embedded line $D\sigma \subset X$ is geodesic.
It suffices to show that the segment $\sigma \cup b\sigma \subset X$ joining $x$ and $bx$ is geodesic.
Let $\tau \subset X$ be the unique geodesic segment joining $x$ and $bx$.
Since $b^2=1$ and the action of $b$ is isometric, $\tau$ is $b$-invariant and its midpoint $m$, with respect to the arclength parameterization, is fixed by $b$.
Hence $y=m$ and so $\sigma \cup b\sigma=\tau$ is geodesic.
Therefore $\ell_D := D \sigma$ is the unique $D$-invariant geodesic line in $X$.

If $D' \in \vc(\G) - \fbc(\G)$ satisfies $D \subseteq D'$, then $\ell_{D'}$ is $D$-invariant, hence $\ell_{D'} = \ell_D$, so $D' \subseteq \Stab_\G(\ell_D)$.
Therefore, since $\Stab_\G(\ell_D)$ has a proper isometric action on $\ell_D \approx \R$,
it is the unique maximal virtually cyclic subgroup of $\G$ containing $D$.
Thus Hypothesis~(2) holds.
Now apply Theorem~\ref{thm:mainQ} in order to obtain the bijection~\eqref{eqn:main}.
\end{proof}

\section{Geometric examples}\label{sec:exm}

Indeed, such $\CAT(0)$ examples of $(X,\G)$ exist which cannot be Riemannian.
A natural source for such infinite $\G$ with $2$-torsion are reflection groups of convex polytopes.
Thanks go to Mike Davis for feedback on this exposition and a guide to define $\Gamma$ below.

Let $K$ be an abstract simplicial complex with finite vertex set $S$.
In \cite[Section~1.2]{Davis_Coxeter}, Davis constructs a cubical cell complex $P_K$ and right-angled Coxeter system $(W_K,S)$:
\begin{eqnarray}
\label{eqn:P} P_K &~:=~& \bigcup_{\sigma \in K} [-1,1]^\sigma \times \{-1,1\}^{S-\sigma} ~\subset~ [-1,1]^S\\
W_K &~=~& \left\langle S ~\mid~ \{s^2 = 1\}_{s \in S} ~,~ \{ [s,t]=1 \}_{\{s,t\} \in K} \right\rangle.
\end{eqnarray}
Herein, we use the set-theoretic notation $B^A := \{ \text{functions } f: A \longrightarrow B\}$.

The link of each vertex of $P_K$, hence of each vertex of the universal cover $\wt{P_K}$, is isomorphic to the geometric realization $|K| \subset [0,1]^S$.
There is a cocompact, proper, isometric action $W_K \curvearrowright \widetilde{P_K}$ covering the natural reflection action $W_K \curvearrowright [-1,1]^S$.
From these actions, Davis obtains an identification and an exact sequence of groups:
\begin{equation}\label{eqn:abelianization}\begin{CD}
1 @>>> \pi_1(P_K) = [W_K,W_K] @>>> W_K @>{\varphi}>> \{-1,1\}^S @>>> 1.
\end{CD}
\end{equation}

The \emph{barycentric subdivision} $bK$ is the abstract simplicial complex whose $n$-simplices are all linearly ordered subsets of $K$ of cardinality $n+1$.
A simplicial complex is \emph{flag} if, whenever a finite subset of vertices are pairwise joined by edges, they span a simplex.
Since $bK$ is flag, by \cite[Proposition~1.2.3]{Davis_Coxeter}, the induced metric on $X := \widetilde{P_{bK}}$ is $\CAT(0)$.
Then, since $P_{bK}$ is aspherical, \eqref{eqn:abelianization} implies that $W := W_{bK}$ is virtually torsion-free.

\begin{lmm}\label{lem:gamma}
Let $K$ be an abstract simplicial complex with finite vertex set.
Recall the right-angled Coxeter group $W$ and the cubical complex $X$ defined above.
There is a virtually torsion-free subgroup $\G \trianglelefteq W$ with torsion such that $\G \curvearrowright X$ is pseudo-free.
\end{lmm}

\begin{proof}
Note $bK$ has vertex set $K$. Write $n := \dim K$.
Consider the epimorphism
\[
\theta: \{-1,1\}^K \longrightarrow \{-1,1\}^{n+1} ~;~\quad
f \longmapsto \left(\prod_{\dim\sigma=i} f(\sigma)\right)_{i=0}^n.
\]
Define a normal subgroup
\[
\G ~:=~ (\theta\circ\varphi)^{-1}\langle(-1, \ldots, -1)\rangle ~\trianglelefteq~ W.
\]
Note $\G$ is virtually torsion-free: in fact, \eqref{eqn:abelianization} restricts to an exact sequence
\[\begin{CD}
1 @>>> [W,W] @>>> \G @>{\varphi}>> \theta^{-1}\langle(-1, \ldots, -1)\rangle @>>> 1.
\end{CD}\]

Observe the reflection action $W_{\Delta^n} = \{-1,1\}^{n+1} \curvearrowright [-1,1]^{n+1} = P_{\Delta^n}$ restricts to a pseudo-free action $\langle(-1, \ldots, -1)\rangle \curvearrowright P_{\Delta^n}$.
There is a cubical map $P_{bK} \to P_{\Delta^K} \to P_{\Delta^n}$, induced on $P$-constructions by an inclusion and a projection, that is equivariant with respect to the homomorphism $\theta \circ \varphi: W \longrightarrow W_{\Delta^n}$ and is injective on each cube $[-1,1]^\sigma$.
Then $W \curvearrowright P_{bK}$ restricts to a pseudo-free action $\G \curvearrowright P_{bK}$.
So, since the map $X \longrightarrow P_{bK}$ is $W$-equivariant and is injective on each cube, the action $\Gamma \curvearrowright X$ is pseudo-free.
\end{proof}

%

\begin{exa}\label{exa:Davis}
Now we proceed to specify the exotic $\CAT(0)$ examples $W \curvearrowright X$ of Davis--Januskiewicz, recounted in \cite[Example~10.5.3]{Davis_Coxeter}.
The key feature is that $X$ is a topological manifold of any given dimension $n \geq 7$, but it not simply connected at infinity.
Hence $X$ is a contractible $n$-dimensional manifold, not homeomorphic to $\R^n$.

Let $3 \leq m \leq n - 4$.
Start with a triangulated homology $m$-sphere $M$ with fundamental group $\pi \neq 1$.
(Recall a \emph{homology $m$-sphere} is a closed manifold with the same integral homology groups as $S^m$.)
For example, $M$ can be the Poincar\'e homology 3-sphere.
Write $C$ for the complement of the open star of a vertex in $M$.
Then $C$ is a compact, triangulated $\partial$-manifold of dimension $m$, with the fundamental group $\pi$, and $\bdry C \approx S^{m-1}$.
Thicken $C$ into a compact, triangulated $\partial$-manifold
\[
A := C \times D^{n-m-1} ~\text{ with }~ \bdry A ~\approx~ (C \times S^{n-m-2}) \cup_{(S^{m-1} \times S^{n-m-2})} (S^{m-1} \times D^{n-m-1}).
\]
Note the induced map $\pi_1(\bdry A) \to \pi_1(A) = \pi$ of fundamental groups is an isomorphism.
Furthermore, $\bdry A$ is a homology $(n-2)$-sphere, since $M$ is a homology $m$-sphere.

Define a simply connected homology-manifold $L$ of dimension $n-1$ by
\[
L ~:=~ A \cup_{\bdry A} \mathrm{Cone}(\bdry A).
\]
Observe that $L$ is not a manifold since the link of the cone point $c$ is not a sphere.
Nonetheless, by a theorem of Edwards, the suspension of $L$ is a topological manifold.
More generally, this is true for any triangulated homology-manifold with simply connected links.
Write $K$ for the abstract simplicial complex of $L$.
Consider the cubical Davis complex $X$, right-angled Coxeter system $(W,S)$, and subgroup $\G$ from Lemma~\ref{lem:gamma}.
Each vertex link is $|bK| \approx |K| = L$.
Thus $X$ is a topological manifold.
Therefore, Corollary~\ref{cor:CAT0} calculates $\fS(\G)$.
But, by \cite[Theorem~9.2.2]{Davis_Coxeter}, $\pi_1(L-c)\neq 1$ implies $\pi_1^\infty(X) \neq 1$.
\end{exa}

Finally, the axiomatic formulation of Theorem~\ref{thm:mainQ} is worthwhile;
it removes the reliance on convex geometry in the proof.
Here is a non-convex example to illustrate the axioms; thanks go to David Speyer for pointing it out on \texttt{http://mathoverflow.net}.

\begin{exa}\label{exa:Heisenberg}
For any commutative ring $R$, recall the $R$-Heisenberg group
\[
\Hei(R) ~:=~ \left\{ \begin{pmatrix}1 & x & z\\ 0 & 1 & y\\ 0 & 0 & 1\end{pmatrix} ~\bigg\vert~ x,y,z \in R \right\} ~\subset~ GL(3,R).
\]

Consider the Eisenstein integers $\Z[\omega]$, where $\omega := \exp(2\pi i/3) \in \C$ is a primitive third root of unity.
Also consider the diagonal matrix $D := \mathrm{diag}(1,\omega,\omega^2) \in GL(3,\C)$.
Define a semidirect product $\G = \Hei(\Z[\omega]) \rtimes C_3$, where the $C_3$-action is given by conjugation by $D$ in $GL(3,\C)$.
Take $X = \Hei(\C)$.
Then $\G$ satisfies Hypotheses~(1--5), using a theorem of Bartels--Farrell--L\"uck \cite{BFL_lattice}.
Therefore: $\fS(\G) = \{ [X,\G] \}$, by Theorem~\ref{thm:mainQ}.

Recall the Solvable Subgroup Theorem \cite[II.7.8]{BH_CAT0}:
if a virtually solvable group $\G$ admits a cocompact proper action by isometries on a $\CAT(0)$ space, then $\G$ must be virtually abelian.
However, our group $\G$ is virtually solvable but not virtually abelian.
Therefore our $\G$ cannot act cocompactly and properly by isometries on a $\CAT(0)$ space.
\end{exa}


\end{document}